\documentclass{article}

\usepackage{amsmath} 
\usepackage{amssymb} 
\usepackage{hyperref}
\usepackage{amsthm}
\usepackage{adjustbox}
\usepackage{float}
\usepackage{indentfirst}
\usepackage{enumerate,longtable, tabu}
\graphicspath{{./img/}{./figures}}
\usepackage{tkz-euclide}
\usepackage{tkz-berge}
\usepackage{xparse}
\usepackage{amsmath,amssymb}
\usetikzlibrary{topaths,calc,decorations.pathreplacing,hobby,fit,scopes,matrix,positioning,decorations.pathmorphing,automata,backgrounds}

\tikzset{
 dot/.style={draw, circle, inner sep=0.5pt, fill=blue}
}

\ExplSyntaxOn
\NewDocumentCommand \IfValueInList {mmmm}{
\clist_if_in:nVTF {#2} #1 {#3} {#4}
}
\cs_set_eq:NN \IfEmptyTF \tl_if_blank:nTF
\ExplSyntaxOff

\DeclareDocumentCommand\longpath{m m m g g}{
\def \n {#3}  \def \startn {#1}  \def \endn {#2}
\IfNoValueTF{#5}{\def \p {p}}{\def \p {#5} }
\foreach \v in {0,...,\n} {\coordinate (\p\v) at ($(\startn)!\v/\n!(\endn)$);}
\IfNoValueTF{#4}
{\foreach \x [evaluate=\x as \y using int(\x-1)] in {1,...,\n}{\draw[black](\p\y) node {}--(\p\x) node {};}}
{\foreach \x [evaluate=\x as \y using int(\x-1)] in {1,...,\n}{
\IfValueInList \x {#4}
{\foreach \v in {4,5,6}{\fill ($ (\p\y)!\v/10! (\p\x) $) circle (0.5pt);}}
{\draw[black](\p\y) node {}--(\p\x) node {};}
}}}

\DeclareDocumentCommand\nstar{o m m m m g g}{
\def \o {#2}  \def \n {#3}  \def \deg {#4}  \def \len {#5}
\IfNoValueTF{#1}{\def \rot {0}}{\def \rot {#1} }
\IfNoValueTF{#7}{\def \p {p}}{\def \p {#7} }
\foreach \x in {0,1,...,\n} {\coordinate[rotate around={-0.5*\deg+\rot:(\o)}](\p\x) at ($(\o)+({\deg * \x/\n}:\len)$);}
\foreach \x[evaluate=\x as \y using int(\x+1),evaluate=\x as \z using int(\x-1)] in {0,1,...,\n}{
  \IfValueInList \x {#6}{\foreach \v in {2,4,6}{
  \fill[black] ($ (\p\z)!\v/8! (\p\y) $) circle (0.5pt);
   }}
  {\draw (\o) node {}-- (\p\x)node{};}
}
}

\DeclareDocumentCommand\star{o m m m m g g}{
\def \o {#2}  \def \n {#3}  \def \deg {#4}  \def \len {#5}
\IfNoValueTF{#1}{\def \rot {0}}{\def \rot {#1} }
\IfNoValueTF{#7}{\def \p {p}}{\def \p {#7} }
\foreach \x in {0,1,...,\n} {\coordinate[rotate around={-0.5*\deg+\rot:(\o)}](\p\x) at ($(\o)+({\deg * \x/\n}:\len)$);}
\foreach \x[evaluate=\x as \y using int(\x+1),evaluate=\x as \z using int(\x-1)] in {0,1,...,\n}{
  {\draw (\o) node {}-- (\p\x)node{};}
}
}

\DeclareDocumentCommand\fitellipsis{o m m m m}{
\IfNoValueTF{#1}{\def \fc {gray}}{\def \fc {#1}}
\draw[fill=\fc!25, fill opacity=0.5] let \p1=(#2), \p2=(#3), \n1={atan2(\y2-\y1,\x2-\x1)}, \n2={veclen(\y2-\y1,\x2-\x1)} in ($ (\p1)!0.5!(\p2) $) ellipse [radius=\n2/2+#4pt, y radius=#5pt, rotate=\n1];
\foreach \v in {4,5,6} {
\fill ($ (#2)!\v/10!(#3) $) circle (0.5pt);
}
}

\newcommand{\dotellipsis}[4] 
{\draw[dashed] let \p1=(#1), \p2=(#2), \n1={atan2(\y2-\y1,\x2-\x1)}, \n2={veclen(\y2-\y1,\x2-\x1)}
    in ($ (\p1)!0.5!(\p2) $) ellipse [radius=\n2/2+#3pt, y radius=#4pt, rotate=\n1];
}

\DeclareDocumentCommand\hyperpath{o m m m }{
\IfNoValueTF{#1}{\def \fco {gray}}{\def \fco {#1}}
\pgfmathtruncatemacro{\len}{#4}
\foreach \v in {0,...,\len} {
	\coordinate (p\v) at ($(#2)!\v/\len!(#3)$);
	 }

\pgfmathtruncatemacro{\lena}{\len-1}
\foreach \cur  [count=\next from 1] in {0,...,\lena} {
	\fitellipsis[\fco]{p\cur}{p\next}{8}{6};
}
\foreach \v in {0,...,\len} {
   \node[shade,shading=ball,circle,ball color=red!90!white,minimum size=1.8pt,inner sep=0pt] at  (p\v) {};
	 }
}

\newcommand{\dotle}{
  \centering
  \begin{tikzpicture}
  \draw (0,0) node[inner sep = 0pt] {$<$};
  \fill (0.1,0) circle (0.02); 
  \end{tikzpicture}%
}

\newtheorem{lem}{Lemma}[section]

\newtheorem{proof*}{proof}[section]
\newtheorem{cor}{Corollary}[section]

\newtheorem{question}{Question}

\newtheorem{remark}{Remark}

\newtheorem{theorem}{Theorem}[section]

\newtheorem{corollary}[theorem]{Corollary}

\newtheorem{example}{Example}

\newcommand{\x}{{\ensuremath{ \boldsymbol{x}}}}
\newcommand{\y}{{\ensuremath{ \mathbf{y}}}}
\newcommand{\z}{{\ensuremath{ \mathbf{z}}}}

\DeclareMathOperator{\spec}{spec}

\title{Exploring Graphs with Distinct $M$-Eigenvalues: Product Operation, Wronskian Vertices, and Controllability\footnote{
   E-mail: shan\_haiying@tongji.edu.cn (H. Shan), 2111152@tongji.edu.cn (X. Liu) }}

\author{Haiying Shan\thanks{Corresponding author}, \, Xiaoqi Liu
\\ \small School of Mathematical Sciences, Key Laboratory of Intelligent Computing and Applications\\ \small(Ministry of Education), Tongji University, Shanghai, China}

\begin{document}
\maketitle

\begin{abstract}
Let $\mathcal{G}^M$ denote the set of connected graphs with distinct $M$-eigenvalues. This paper explores the $M$-spectrum and eigenvectors of a new product $G\circ_C H$ of graphs $G$ and $H$. We present the necessary and sufficient condition for $G\circ_C H$ to have distinct $M$-eigenvalues. Specifically, for the rooted product $G\circ H$, we present a more concise and precise condition. A key concept, the $M$-Wronskian vertex, which plays a crucial role in determining graph properties related to separability and construction of specific graph families, is investigated. We propose a novel method for constructing infinite pairs of non-isomorphic $M$-cospectral graphs in $\mathcal{G}^M$ by leveraging the structural properties of the $M$-Wronskian vertex. Moreover, the necessary and sufficient condition for $G\circ H$ to be $M$-controllable is given.
\end{abstract}
\vspace{4mm}  	


\section{Introduction}
Let \( G \) be a simple graph with \( n \) vertices, and let \( M(G) \) be a real symmetric matrix associated with \( G \). The characteristic polynomial of \( M(G) \), defined as \( \det(xI - M(G)) \), is called the \textbf{\( M \)-characteristic polynomial} and is denoted by \( \phi_M(G, x) \) or \( \phi(M(G), x) \). The eigenvalues of \( M(G) \) are referred to as the \textbf{\( M \)-eigenvalues}, and their multiset is the \textbf{\( M \)-spectrum}, denoted by \( \mathrm{Spec}_M(G) \). The spectrum of a matrix $M$ is denoted as $\mathrm{Spec}(M)$.

For specific choices of \( M(G) \), such as the adjacency matrix \( A \), Laplacian matrix \( L \), normalized Laplacian matrix \( Q \), or \( A_\alpha \) matrix, the corresponding eigenvalues are termed the \( A \)-spectrum, \( L \)-spectrum, \( Q \)-spectrum, and \( A_\alpha \)-spectrum, respectively.
The universal adjacency matrix of $G$ is defined as $U(G)=a A(G)+d D(G) +i I_n+j J_n$, where $I_n$ represents the identify matrix, $J_n$ represents the all-ones matrix, and $a\neq 0$, $d,i,j$ are all constant (see \cite{HAEMERS20112520}), we only consider the situation that $i=j=0$ in this paper. In the remaining part of this article, we assume by default that $M\in\{A,L,Q,A_\alpha,U\}$.

Let $A$ and $B$ be $m\times n$ and $p\times q$ matrices, respectively. The \textbf{Kronecker product} of $A$ and $B$, which is denoted as $A\otimes B$, is the $mp\times nq$ matrix as follows:
$$\begin{bmatrix}
	a_{11}B & a_{12}B &\cdots & a_{1n}B\\
	a_{21}B & a_{22}B & \cdots & a_{2n}B\\
	\vdots & \vdots & \ddots & \vdots\\
	a_{m1}B & a_{m2}B & \cdots & a_{mm}B
\end{bmatrix}
$$

Let \( G \) be a graph with vertex set \( V(G) = \{v_i\}_{i=1}^n \), and let \( H \) be a rooted graph with root \( u \). The \textbf{ rooted product} \( G \circ H \) with respect to the ``root'' $u$ is defined as follows: Take $n$
copies of $H$, and for every vertex $v_i$ of $G$, identify $v_i$
with the root $u$ of the $i$th copy of $H$.  For more details, see Godsil and MacKay~\cite{GM78} and Schwenk and Allen~\cite{computing}.
 
Here we define another product of graphs. Let $G=(V(G), E(G))$ and $H=(V(H), E(H))$ be two graphs, where $V(G)=\{u_1, u_2, \cdots, u_{n}\}$ and $V(H)=\{v_1, v_2, \cdots, v_{m}\}$, C is a (0,1)-symmetric matrix of order $m$. We define $G\circ_{C} H$ as the graph with vertex set $V(G\circ_C H)=\{(u_i,v_j)|u_i\in V(G), v_j\in V(H)\}$. The vertices $(u_{i_1},v_{j_1})$ and $(u_{i_2},v_{j_2})$ are adjacent in $G\circ_{C} H$ if either $u_{i_1}$ and $u_{i_2}$ are adjacent in $G$ and $c_{j_1j_2}=1$, or $v_{j_1}$ and $v_{j_2}$ are adjacent in $H$ and $i_1=i_2$ (see Figure 1 for example). 

\begin{figure}[H]
	\centering
	\begin{tikzpicture}[scale = 2.5,
		Node/.style = {fill,circle,scale =.5}]
		\node[Node,label = {[label distance=-2pt,above]0: $(u_1,v_1)$}] (1)  at (0, 1) {};
		\node[Node,label = {[label distance=-2pt,above]0: $(u_2,v_1)$}] (2)  at (1, 1) {};
		\node[Node,label = {[label distance=-2pt,above]0: $(u_3,v_1)$}] (3)  at (2, 1) {};
		\node[Node,label = {[label distance=-2pt,left]0: $(u_1,v_2)$}] (4)  at (0, 0.5) {};
		\node[Node,label = {[label distance=-2pt,left]0: $(u_2,v_2)$}] (5)  at (1, 0.5) {};
		\node[Node,label = {[label distance=-2pt,right]0: $(u_3,v_2)$}] (6)  at (2, 0.5) {};
		\node[Node,label = {[label distance=-2pt,below]0: $(u_1,v_3)$}] (7)  at (0, 0) {};
		\node[Node,label = {[label distance=-2pt,below]0: $(u_2,v_3)$}] (8)  at (1, 0) {};
		\node[Node,label = {[label distance=-2pt,below]0: $(u_3,v_3)$}] (9)  at (2, 0) {};
		\draw[color=black](1)--(4);
		\draw[color=black](2)--(5);
		\draw[color=black](3)--(6);
		\draw[color=black](4)--(7);
		\draw[color=black](5)--(8);
		\draw[color=black](6)--(9);
		\draw[color=black](1)--(2);
		\draw[color=black](3)--(2);
		\draw[color=black](7)--(8);
		\draw[color=black](8)--(9);
		\draw[color=black](4)--(8);
		\draw[color=black](5)--(9);
		\draw[color=black](7)--(5);
		\draw[color=black](6)--(8);
	\end{tikzpicture}
	\caption{Graph $P_3\circ_C P_3$ with $C=\begin{bmatrix}
			1 & 0 & 0\\
			0 & 0 & 1\\
			0 & 1 & 1
		\end{bmatrix}.$}
\end{figure}
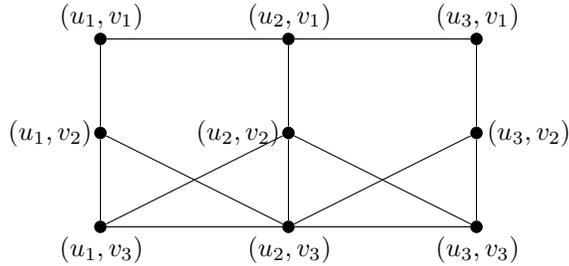

The adjacency matrix of $G\circ_CH$ is given by the formula \begin{equation}\label{eqq1}
    A(G\circ_CH)=C \otimes A(G) + A(H) \otimes I_n.
\end{equation}

Let  $C_1$ be the diagonal matrix such that the $i$-th diagonal element of $C_1$ is equal to the $i$-th row sum of $C$. Then we get 
\begin{equation*}\label{eqq2}
    D(G\circ_C H) = C_{1} \otimes D(G) + D(H) \otimes I_n.
\end{equation*}
When $C$ is not a diagonal matrix, we get $C\neq C_1$. For the universal adjacency matrix $U=aA+dD$, we have 
$$U(G\circ_CH)=a(C \otimes A(G) + A(H) \otimes I_n)+d(C_{1} \otimes D(G) + D(H) \otimes I_n),$$
then $$U(G\circ_CH)-C \otimes U(G) + U(H) \otimes I_n =d((C_{1}-C) \otimes D(G) )\neq 0.$$
When $C$ is a diagonal matrix, we get $C=C_1$, then for $M\in \{L,Q,A_\alpha,U\}$, we have 
$$M(G\circ_CH)=C \otimes M(G) + M(H) \otimes I_n.$$
Note that when $C=E_{i,i}$, $G\circ_C H$ turns out to be the rooted product of $G$ and $H$ where $u$ is the root vertex of $H$, which corresponds to the $i$-th row of $M(H)$. When $C=I_m$, $G\circ_C H$ turns into the Cartesian product of $G$ and $H$ (see \cite{MR209178}).

For each edge $e$ of $G$, let there be associated a real number $w_G(e)$, called its \textbf{weight}. Then $G$, together with these weights on its edges, is called a \textbf{weighted graph}, and denoted $(G,w_G)$ (or simply $G$). The weighted adjacency matrix $A(G)$ of $G$ is defined as $a_{ij}=w_G((v_i,v_j))$.
And the degree matrix $D(G)$ of $G$ is defined as 
$$d_{ij}=\begin{cases}
    \sum_k w_G((i,k))  &\text{if}\quad i=j\\
    0\quad &\text{otherwise}
\end{cases}$$

Let \(G\) and \(H\) be two weighted graphs with weighted adjacency matrices, respectively. Given a real symmetric matrix \(C\),  \(G \circ_C H\) is defined as the weighted graph whose weighted adjacency matrix satisfies:
\[
A(G\circ_C H) = C \otimes A(G) + A(H) \otimes I_n.
\]

A polynomial $ f(x) $ over a field $ K $ is separable if all of its roots in any extension of $ K $ are distinct. That is, for every root $ \alpha $ of $ f(x) $, we have $ f'(\alpha) \neq 0 $, meaning the derivative $ f'(x) $ does not vanish at $ \alpha $.
If the greatest common divisor of $ f(x) $ and its derivative $ f'(x) $ $ \text{gcd}(f, f') = 1 $, then $ f(x) $ is \textbf{separable}. If $ \text{gcd}(f, f') $ is a non-trivial polynomial (i.e., not a constant), then $ f(x) $ is \textbf{inseparable}.
 
A graph $ G $ is called \textbf{$ M $-separable} if the characteristic polynomial of $M(G)$  is separable, which means $G$ has $n$ distinct $M$-eigenvalues. We denote $\mathcal{G}^{M}$ as the set of connected $M$-separable graphs.

Two graphs $G$ and $H$ are called \textbf{$M$-cospectral} if $\mathrm{Spec}_M(G)=\mathrm{Spec}_M(H)$, and a graph $G$ is said to be determined by its $M$-spectrum ($DM S$ for short) if $G\cong H$ whenever $\mathrm{Spec}_M(G)=\mathrm{Spec}_M(H)$ for any graph $H$. 

Given differentiable functions $f_1$ and $f_2$, the \textbf{Wronskian determinant} (or simply the \textbf{Wronskian}) $W(f_1,f_2)$ is defined as the determinant of the following square matrix: $$ W(f_1,f_2)=\begin{vmatrix} f_1 & f_2 \\ f_1' & f_2' \end{vmatrix} $$

The Wronskian can be defined similarly for a set of functions. This determinant is crucial for analyzing the linear dependence of functions, see \cite{MR2115075}. 

Now, we introduce the concept of the \textbf{$M$-Wronskian vertex} of a graph.
Let $H$ be graph and $u$ is a vertex of $H$. $M(H)$ is a real symmetric matrix associated with $H$ and $M^u(H)$ is obtained by deleting the row and the column corresponding to $u$ from $M(H)$. If $$  \phi(M(H), x) \phi'(M^u(H), x)-\phi'(M(H), x)  \phi(M^u(H), x)  \neq 0$$  for all real $x$, we call $u$ the $M$-Wronskian vertex of $H$. Note that when $u$ is the $M$-Wronskian vertex of $H$, we always have $\gcd(\phi(M(H),x),\phi(M^u(H),x)) = 1$ and $H$ is $M$-separable.

Let $G_v^n$ be the graph obtained from $G$ by adding a pendant path of length $n$ at $v$ and denote $u_n$ as the pendant vertex of the pendant path (see Figure 2).

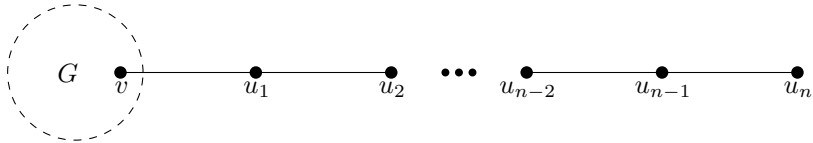
\begin{figure}[H]
\centering
    \begin{tikzpicture}[scale = 3,
        Node/.style = {fill,circle,scale =.5}
        ]
        \node[Node,label = {[label distance=-2pt,below]0: $v$}] (B)  at (0.4, 0) {};
        \node[Node,label = {[label distance=-2pt,below]0: $u_1$}] (3)  at (1, 0) {};
        \node[Node,label = {[label distance=-2pt,below]0: $u_2$}] (4)  at (1.6, 0) {};
        \node[Node,label = {[label distance=-2pt,below]0: $u_{n - 2}$}] (5)  at (2.2, 0) {};
        \node[Node,label = {[label distance=-2pt,below]0: $u_{n - 1}$}] (6)  at (2.8, 0) {};
        \node[Node,label = {[label distance=-2pt,below]0: $u_n$}] (C)  at (3.4, 0) {};
        \coordinate (1) at (0.2,0);
        \draw[dashed] (1) circle[radius = 0.3];
        \node[shift = {(-7mm,0pt)}] at (0.4,0){$G$};
        \longpath{B}{C}{5}{3}{p};
    \end{tikzpicture}
    \caption{Graph $G_v^n$.}
\end{figure}

An eigenvalue is considered main if its corresponding eigenspace contains a vector that is not orthogonal to the all-ones vector. We say a matrix $M$ is \textbf{controllable} if all the eigenvalues of $M$ are main, and a \textbf{$M$-controllable graph} is a connected graph whose $M$-matrix has distinct and main eigenvalues. 
It is well known in control theory (see \cite{chen2013linear}) that $M$ is
controllable if and only if the controllability matrix $C(M)=(e, Me, M^2 e, \dotsc, M^{n-1} e)$ has full rank, where $e$ is all-ones vector.
Cvetkovi{\'c} defined the controllable graph in \cite{zbMATH06289204} and provided some characterizations of the controllable graphs whose least eigenvalue is bounded from below by $-2$ in \cite{zbMATH06124038}. 
For $A=A(G)$, $C(A)$ is called the walk matrix of $G$ and denoted as $W(G)$. So $G$ is controllable if and only if $W(G)$ is full rank. Previous studies have investigated the controllability of specific types of graphs, such as graphs with a diameter of 
$n-2$ \cite{zbMATH07424199} and NEPSes of graphs \cite{zbMATH07549474}.

The study of graphs with distinct eigenvalues can be traced back to the work of Harary and Schwenk in 1974 \cite{Harary1974WhichGH}, where they posed an interesting problem: ``Which graphs have distinct eigenvalues?''. The pioneering result for this problem is from Mowshowitz \cite{mowshowitz1969group} and Chao \cite{CHAO1971301} generalized Mowshowitz's result to more general graph classes. Li et al. \cite{2015On} gave an algebraic characterization of the graphs with distinct eigenvalues. Lou et al. \cite{zbMATH06680902,zbMATH06812655} gave a new method to construct graphs with distinct $A$-eigenvalues and $Q$-eigenvalues, and Tian et al. \cite{zbMATH07794544}  improved and generalized the main results obtained by Lou et al. to $\alpha$-spectrum. The graphs they investigated can be obtained by the rooted product of some graph $G$ and $P_n$.

The remainder of this paper is organized as follows. In Section 2, the spectrum and the eigenvectors of $G\circ_C H$ were studied. In Section 3, our investigation focused on $G\circ_C H$ with distinct $M$-eigenvalues. We investigated the $M$-Wronskian vertex of $G^n_v$ and provided a construction method for infinite graphs with $M$-Wronskian vertex in Section 4. Section 5 is dedicated to studying the $DMS$-property of a rooted product graph and gives a method for constructing infinite pairs of non-isomorphic $M$-cospectral graphs in $\mathcal{G}^M$. Finally, the necessary and sufficient condition for $G\circ H$ to be $M$-controllable is given in Section 6.

\section{The $M$-spectrum and $M$-eigenvectors of $G\circ_C H$}

Godsil and MacKay~\cite{GM78} derived the $A$-characteristic polynomial of $G\circ H$. Moreover, Maghsoudi et al.~\cite{zbMATH06998812} determined the $Q$-characteristic polynomial of $G\circ H$. Furthermore, Rajkumar, R. and Pavithra, R. \cite{MR4364384} explored the spectra of more general rooted product graphs. Now we give the $M$-spectrum and $M$-eigenvector of $G\circ_C H$. We first give a result about the Kronecker product of matrices.

\begin{lem}\label{mlem1}
    Let $A$ and $B$ be two matrices of order $n$ and $m$, $C$ be a matrix of order $m$. Let $B(\mu)=B+\mu C$, define $S(\mu)$ as the multiset consisting of the eigenvalues of $B(\mu)$. Then, the spectrum of the matrix $D = C \otimes A + B \otimes I_n$ is determined as follows:
    $$\mathrm{Spec}(D)=\bigcup_{\mu\in \mathrm{Spec}(A)}S(\mu).$$
    Moreover, let $ \mathrm{Spec}(A) = \{\mu_1, \mu_2, \cdots, \mu_n\} $, and $\xi_1,\xi_2,\cdots,\xi_n$ form an orthonormal basis of $A$. And let $\eta_{1j}, \eta_{2j}, \dots, \eta_{mj} $ form an orthonormal basis of $B(\mu_j)$. Then $ \{\eta_{ij} \otimes \xi_j| i = 1, 2, \cdots, m,  j = 1, 2, \cdots, n \} $ spans the eigenspace of $ D $.
\end{lem}
\begin{proof}
    Let $\lambda\in \mathrm{Spec}(B(\mu))$, and assume $\eta\neq 0$ satisfies that $B(\mu)\eta=\lambda \eta$, and $\xi$ is a eigenvector of $A$ corresponding to $\mu$, then we have
\begin{align*}
D(\eta \otimes \xi) &= (C \otimes A)(\eta \otimes \xi) + (B\otimes I_n)(\eta \otimes \xi) \\
&= (C\eta \otimes A\xi) + (B\eta \otimes \xi) \\
&= (C\eta \otimes \mu \xi) + (B\eta \otimes \xi) \\
&= (\mu C\eta + B\eta) \otimes \xi \\
&= \lambda \eta \otimes \xi.
\end{align*}
Consequently, $\eta \otimes \xi$ represents an eigenvector of $D$ with the eigenvalue $\lambda$.

Consider the set of vectors $ \{\eta_{ij} \otimes \xi_j\} $, where $ i = 1, 2, \dots, m $ and $ j = 1, 2, \dots, n $, we have $(\eta_{i_1j_1} \otimes \xi_{j_1})\otimes(\eta_{i_2j_2}\otimes \xi_{j_2})=(\eta_{i_1j_1} \otimes \eta_{i_2j_2})\otimes (\xi_{j_1}\otimes \xi_{j_2})=0$ if either $i_1\neq i_2$ or $j_1\neq j_2$. And $D(\eta_{ij}\otimes \xi_j)=\lambda_{ij}(\eta_{ij}\otimes \xi_j)$, where $\lambda_{ij}$ is the eigenvalue of $B(\mu_j)$ corresponding to $\eta_{ij}$. Thus, $\mathrm{Spec}(D)=\bigcup_{\mu\in \mathrm{Spec}(A)}S(\mu)$ and $ \{\eta_{ij} \otimes \xi_j\} $ spans the eigenspace of $ D$.
\end{proof}
Let $G$ and $H$ be two graphs of order $n$ and $m$, respectively. $C$ is a (0,1)-symmetric matrix of order $m$, we have $A(G\circ_CH)=C \otimes A(G) + A(H) \otimes I_n$. Then by Lemma \ref{mlem1}, we can immediately get the following result.

\begin{theorem}\label{cthm1}
     Let $G$ and $H$ be two graphs of order $n$ and $m$, respectively. Suppose that $C$ is a (0,1)-symmetric matrix of order $m$. Let $ B(\mu) = A(H) + \mu C $, define $S(\mu)$ as the multiset consisting of the eigenvalues of $B(\mu)$. Then, the $A$-spectrum of $G\circ_C H$ is determined as follows:

$$\mathrm{Spec}_A(G\circ_C H)=\bigcup_{\mu\in \mathrm{Spec}_A(G)}S(\mu).$$
Moreover, let $ \mathrm{Spec}_A(G) = \{\mu_1, \mu_2, \cdots, \mu_n\} $, and $\xi_1,\xi_2,\cdots,\xi_n$ form an orthonormal basis of $A(G)$. And let $\eta_{1j}, \eta_{2j}, \dots, \eta_{mj} $ form an orthonormal basis of $B(\mu_j)$. Then $ \{\eta_{ij} \otimes \xi_j| i = 1, 2, \cdots, m,  j = 1, 2, \cdots, n \} $ spans the eigenspace of $ A(G\circ_C H) $.
\end{theorem}
When $C$ is a (0,1)-diagonal matrix, we have $C=C_1$, then $M(G\circ_CH)=C \otimes M(G) + M(H) \otimes I_n$, from Lemma \ref{mlem1} we have:
\begin{theorem}\label{cthm2}
     Let $G$ and $H$ be two graphs of order $n$ and $m$, respectively. Suppose that $C$ is a (0,1)-diagonal matrix of order $m$, and $M(G)$ is a real symmetric matrix associated with $G$. Let $ B(\mu) = M(H) + \mu C $, and define $S(\mu)$ as the multiset consisting of the eigenvalues of $B(\mu)$. Then, the $M$-spectrum of $G\circ_C H$ is determined as follows:

$$\mathrm{Spec}_M(G\circ_C H)=\bigcup_{\mu\in \mathrm{Spec}_M(G)}S(\mu).$$
Moreover, let $ \mathrm{Spec}_M(G) = \{\mu_1, \mu_2, \cdots, \mu_n\} $, and $\xi_1,\xi_2,\cdots,\xi_n$ form an orthonormal basis of $M(G)$. And let $\eta_{1j}, \eta_{2j}, \dots, \eta_{mj} $ form an orthonormal basis of $B(\mu_j)$. Then $ \{\eta_{ij} \otimes \xi_j| i = 1, 2, \cdots, m,  j = 1, 2, \cdots, n \}$ spans the eigenspace of $ M(G\circ_C H) $.
\end{theorem}

In the context of the root product $G\circ H$, assume that the vertex $u$ corresponds to the first row of $M(H)$, then $$|xI-B(\mu)|=\phi(M(H),x)-\mu \phi(M^u(H),x),$$ where $ B(\mu) = M(H) + \mu E_{1,1} $. 
Subsequently, applying Theorem \ref{cthm2}, we obtain the following result:
\begin{corollary}\label{thm:spec_rg}
Let $G$ be a graph of order $n$, and $H$ be a rooted graph of order $m$ with root vertex $u$. $M(G)$ is a real symmetric matrix associated with $G$.  Define $S(\mu)$ as the multiset consisting of the roots of $\phi(M(H),x)-\mu \phi(M^u(H),x)$. Then, the $M$-spectrum of $G\circ H$ is determined as follows: 

$$\mathrm{Spec}_M(G\circ H)=\bigcup_{\mu\in \mathrm{Spec}_M(G)}S(\mu).$$
Furthermore, suppose $u$ corresponds to the first row of $M(H)$ and $B(\mu_j)=M(H)+\mu E_{1,1}$. Let $ \mathrm{Spec}_M(G) = \{\mu_1, \mu_2, \cdots, \mu_n\} $, and $\xi_1,\xi_2,\cdots,\xi_n$ form an orthonormal basis of $M(G)$. And let $\eta_{1j}, \eta_{2j}, \dots, \eta_{mj} $ form an orthonormal basis of $B(\mu_j)$. Then $ \{\eta_{ij} \otimes \xi_j| i = 1, 2, \cdots, m,  j = 1, 2, \cdots, n \} $ spans the eigenspace of $ M(G\circ H) $.
\end{corollary}


\section{$G\circ_C H$ with distinct $M$-eigenvalues}

Lou et al. \cite{zbMATH06680902,zbMATH06812655} have shown that $G\circ P_n$ has distinct $A$-eigenvalues, $Q$-eigenvalues if and only if $G$ has distinct $A$-eigenvalues, $Q$-eigenvalues, respectively. Thereafter, Tian et al. \cite{zbMATH07794544} generalized Lou et al.'s result to the $A_\alpha$-spectrum. We further investigate $G\circ_C H$ with distinct $M$-eigenvalues in this section. Based on Theorem \ref{cthm1}, the following result is obvious.
\begin{theorem}
   Let $G$ and $H$ be two graphs of order $n$ and $m$, respectively. Suppose that $C$ is a (0,1)-symmetric matrix of order $m$. Let $ B(\mu) = A(H) + \mu C $, define $S(\mu)$ as the multiset consisting of the eigenvalues of $B(\mu)$. Then $G\circ_C H$ is $A$-separable if and only if both of the following two conditions hold
    \begin{enumerate}[(1).]
        \item For any $\mu\in \mathrm{Spec}_A(G)$, $B(\mu)$ has distinct eigenvalues,
        \item For any two distinct  $\mu_1,\mu_2\in \mathrm{Spec}_A(G)$, $S(\mu_1)\cap S(\mu_2)=\emptyset$.
    \end{enumerate}
\end{theorem}
Similarly, we can immediately obtain the following result according to Theorem \ref{cthm2}.
\begin{theorem}\label{1thm1}
     Let $G$ and $H$ be two graphs of order $n$ and $m$, respectively. Suppose that $C$ is a (0,1)-diagonal matrix of order $m$. Let $ B(\mu) = M(H) + \mu C $, define $S(\mu)$ as the multiset consisting of the eigenvalues of $B(\mu)$. Then $G\circ_C H$ is $M$-separable if and only if both of the following two conditions hold
    \begin{enumerate}[(1).]
        \item For any $\mu\in \mathrm{Spec}_M(G)$, $B(\mu)$ has distinct eigenvalues,
        \item For any two distinct  $\mu_1,\mu_2\in \mathrm{Spec}_M(G)$, $S(\mu_1)\cap S(\mu_2)=\emptyset$.
    \end{enumerate}
\end{theorem}

Next, we focus on the root product, from Corollary \ref{thm:spec_rg}, the following result holds.
\begin{lem}\label{thm1}
    Let $G$ and $H$ be two graphs. Then $S(\mu_1)\bigcap S(\mu_2)=\emptyset$ for any two distinct  $\mu_1,\mu_2\in \mathrm{Spec}_M(G)$ if and only if $(\phi(M(H),x), \phi(M^u(H),x))=1$.
\end{lem}
\begin{proof}
    Assume that $(\phi(M(H),x),\phi(M^u(H),x)) = 1$. Let $\lambda\in S(\mu_1)\cap S(\mu_2)$ for some $\mu_1\neq\mu_2$. Then we have 
\begin{align*}
\phi(M(H),\lambda)-\mu_1\phi(M^u(H),\lambda) &= 0,\\
\phi(M(H),\lambda)-\mu_2\phi(M^u(H),\lambda) &= 0.
\end{align*}
Subtracting these two equations gives $(\mu_1-\mu_2)\phi(M^u(H),\lambda)=0$. Since $\mu_1\neq\mu_2$, we obtain $\phi(M^u(H),\lambda)=0$. Then substituting this into $\phi(M(H),\lambda)-\mu_1\phi(M^u(H),\lambda)=0$ implies that $\phi(M(H),\lambda)=0$. So $x-\lambda$ divides both $\phi(M(H),x)$ and $\phi(M^u(H),x)$, that is, $(x-\lambda)|\gcd(\phi(M(H),x),\phi(M^u(H),x))$, a contradiction.

Now, assume that for any $\mu_1\neq\mu_2$ with $\mu_1,\mu_2\in \mathrm{Spec}_M(G)$, we have $S(\mu_1)\cap S(\mu_2)=\emptyset$.
Suppose that there exists a $\lambda$ such that $(x-\lambda)|\gcd(\phi(M(H),x), \phi(M^u(H),x))$. Then $\lambda\in S(\mu)$ for any $\mu \in \mathrm{Spec}_M(G)$, which implies that $\lambda\in S(\mu_1)\cap S(\mu_2)$, a contradiction.
\end{proof}

Consider the derivative of $\phi(M(H),x)$ and $\phi(M^u(H),x)$. If we desire that the equation $\phi(M(H),x)-\mu\phi(M^u(H),x)=0$ has distinct eigenvalues,  then we should have $\gcd(\phi(M(H),x)-\mu\phi(M^u(H),x),(\phi(M(H),x)-\mu\phi(M^u(H),x))')=1$. Suppose $x-\lambda|\gcd(\phi(M(H),x)-\mu\phi(M^u(H),x),(\phi(M(H),x)-\mu\phi(M^u(H),x))')$. Then we have \begin{align*}
    \phi(M(H),\lambda)-\mu\phi(M^u(H),\lambda) &= 0,\\
   \phi'(M(H),\lambda)-\mu\phi'(M^u(H),\lambda) &= 0.
\end{align*}
From these equations, we can obtain that $$\phi(M(H), \lambda) \phi'(M^u(H), \lambda)-\phi'(M(H), \lambda)  \phi(M^u(H), \lambda)=0.$$
Moreover, if
$\phi(M^u(H),\lambda)\neq 0$, we can get that $\mu=\frac{\phi(M(H),\lambda)}{\phi(M^u(H),\lambda)}$. Based on the above analysis, we have the following result.

\begin{lem}\label{thm2}
Let \(G\) and \(H\) be two graphs. Let \(\mathcal{R}\) be the set of all real roots of the equation  
\[
\phi(M(H), x) \cdot \phi'(M^u(H), x) - \phi'(M(H), x) \cdot \phi(M^u(H), x) = 0.
\]
Then \(G \circ H\) is \(M\)-separable if and only if the following conditions hold:
\begin{enumerate}[(1)]
    \item $G$ is $M$-separable and \(\gcd(\phi(M(H), x), \phi(M^u(H), x)) = 1\),
    \item \(\mathrm{Spec}_M(G) \cap \left\{\frac{\phi(M(H), x)}{\phi(M^u(H), x)} \mid x \in \mathcal{R}\right\} = \emptyset\).
\end{enumerate}
\end{lem}

\begin{proof}
Assume that conditions (1) and (2) hold. By Lemma \ref{thm1}, it suffices to show that for any \(\mu \in \mathrm{Spec}_M(G)\), the polynomial \(\phi(M(H), x) - \mu \phi(M^u(H), x)\) has no multiple roots.  

Suppose, for contradiction, that there exists \(\mu_1 \in \mathrm{Spec}_M(G)\) such that \(\lambda\) is a multiple root of \(\phi(M(H), x) - \mu_1 \phi(M^u(H), x)\). Since \(\gcd(\phi(M(H), x), \phi(M^u(H), x)) = 1\), it follows that \(\phi(M^u(H), \lambda) \neq 0\).  

As \(\lambda\) is a multiple root, we have:
\[
\phi(M(H), \lambda) - \mu_1 \phi(M^u(H), \lambda) = 0,
\]
\[
\phi'(M(H), \lambda) - \mu_1 \phi'(M^u(H), \lambda) = 0.
\]
Eliminating \(\mu_1\) gives:
\[
\phi'(M(H), \lambda) \phi(M^u(H), \lambda) - \phi(M(H), \lambda) \phi'(M^u(H), \lambda) = 0,
\]
which implies \(\lambda \in \mathcal{R}\). 

Moreover, using the first equation, we deduce:
\[
\mu_1 = \frac{\phi(M(H), \lambda)}{\phi(M^u(H), \lambda)}.
\]
This contradicts condition (2), which ensures that \(\mathrm{Spec}_M(G) \cap \left\{\frac{\phi(M(H), x)}{\phi(M^u(H), x)} \mid x \in \mathcal{R}\right\} = \emptyset\).  

Conversely, assume that \(G \circ H\) is \(M\)-separable. By Lemma \ref{thm1}, condition (1) holds. Suppose, for contradiction, that there exists \(\lambda \in \mathcal{R}\) such that:
\[
\frac{\phi(M(H), \lambda)}{\phi(M^u(H), \lambda)} = \mu_1 \in \mathrm{Spec}_M(G).
\]
Then \(\lambda\) is a multiple root of \(\phi(M(H), x) - \mu_1 \phi(M^u(H), x)\). According to Corollary \ref{thm:spec_rg}, this contradicts the \(M\)-separability of \(G \circ H\).
Thus, condition (2) holds, completing the proof.
\end{proof}

By Lemma \ref{thm2}, we can immediately obtain
\begin{cor}\label{cor1}
    Let $G$ be a $ M $-separable graph and $H$ be a rooted graph with $u$ as the root vertex of $H$. Then $G\circ H$ is $ M $-separable if $u$ is the $M$-Wronskian vertex of $H$. 
\end{cor}

\begin{figure}[h]
\centering
		\begin{tikzpicture}[scale=2,
			Node/.style={fill,circle,scale=.5}
			]

        \begin{scope}[xshift=-0.8cm]
        
		\node[Node,label={[label distance=0pt]0: $v_{1}$}] (0)  at (1.5, 2.5) {};
        \node[Node,label={[label distance=-2pt,left]0: $v_{2}$}] (1)  at (1,2.0 ) {};
        \node[Node,label={[label distance=0pt]0: $v_{3}$}] (2)  at (1.5, 2.0) {};
        \node[Node,label={[label distance=-2pt,left]0: $v_{4}$}] (3)  at (1, 1.5) {};
        \node[Node,label={[label distance=0pt]0: $v_{5}$}] (4)  at (1.5, 1.5) {};
        \node[Node,label={[label distance=0pt]0: $v_{6}$}] (5)  at (1, 1) {};
        \node[label={[label distance=0pt]0: $H_1$}] (6)  at (1.2, 0.8) {};
			\draw[color=black](0)--(1);
			\draw[color=black](1)--(2);
			\draw[color=black](2)--(3);
			\draw[color=black](3)--(4);
			\draw[color=black](0)--(2);
			\draw[color=black](1)--(3);
			\draw[color=black](2)--(4);
			\draw[color=black](4)--(2);
                \draw[color=black](4)--(5);
                \draw[color=black](3)--(5);

        \end{scope}
        \begin{scope}[xshift=0cm]
        
        \node[Node,label={[label distance=0pt]0: $v_{1}$}] (7)  at (2.5, 2.5) {};
        \node[Node,label={[label distance=0pt]0: $v_{2}$}] (8)  at (3,2 ) {};
        \node[Node,label={[label distance=-2pt,left]0: $v_{3}$}] (9)  at (2, 2) {};
        \node[Node,label={[label distance=0pt]0: $v_{4}$}] (10)  at (2.5, 1.5) {};
        \node[Node,label={[label distance=0pt]0: $v_{5}$}] (11)  at (3, 1) {};
        \node[Node,label={[label distance=-2pt,left]0: $v_{6}$}] (12)  at (2, 1) {};
        \node[label={[label distance=0pt]0: $H_2$}] (13)  at (2.3, 0.8) {};
			\draw[color=black](7)--(8);
			\draw[color=black](7)--(9);
			\draw[color=black](8)--(9);
			\draw[color=black](10)--(7);
			\draw[color=black](10)--(8);
			\draw[color=black](10)--(11);
			\draw[color=black](10)--(12);
			\draw[color=black](8)--(11);
                \draw[color=black](9)--(12);
                \draw[color=black](11)--(12);
        
        \end{scope}
        \begin{scope}[xshift=0.8cm]
        
        \node[Node,label={[label distance=0pt]0: $v_{1}$}] (21)  at (4, 2.5) {};
        \node[Node,label={[label distance=0pt]0: $v_{2}$}] (22)  at (4.5,2 ) {};
        \node[Node,label={[label distance=-2pt,left]0: $v_{3}$}] (23)  at (3.5, 2) {};
        \node[Node,label={[label distance=0pt]0: $v_{4}$}] (24)  at (4.5, 1.5) {};
        \node[Node,label={[label distance=-2pt,left]0: $v_{5}$}] (25)  at (3.5, 1.5) {};
        \node[Node,label={[label distance=0pt]0: $v_{6}$}] (26)  at (3.5, 1) {};
        \node[label={[label distance=0pt]0: $H_3$}] (27)  at (3.8, 0.8) {};
			\draw[color=black](21)--(22);
			\draw[color=black](21)--(23);
			\draw[color=black](22)--(23);
			\draw[color=black](23)--(25);
			\draw[color=black](24)--(22);
			\draw[color=black](24)--(25);
			\draw[color=black](26)--(25);

        \end{scope}
        \begin{scope}[xshift=1.5cm]
        
        \node[Node,label={[label distance=0pt]0: $v_{1}$}] (31)  at (5.5, 2.5) {};
        \node[Node,label={[label distance=-2pt,left]0: $v_{2}$}] (32)  at (5,2 ) {};
        \node[Node,label={[label distance=-2pt,below]0: $v_{3}$}] (33)  at (5.5, 2) {};
        \node[Node,label={[label distance=0pt]0: $v_{4}$}] (34)  at (6, 2) {};
        \node[Node,label={[label distance=-2pt,left]0: $v_{5}$}] (35)  at (5, 1.5) {};
        \node[Node,label={[label distance=0pt]0: $v_{6}$}] (36)  at (5.5, 1.5) {};
        \node[Node,label={[label distance=0pt]0: $v_{7}$}] (37)  at (5, 1) {};
        \node[label={[label distance=0pt]0: $H_4$}] (38)  at (5.3, 0.8) {};
			\draw[color=black](31)--(32);
			\draw[color=black](31)--(33);
			\draw[color=black](31)--(34);
			\draw[color=black](32)--(33);
			\draw[color=black](33)--(34);
			\draw[color=black](32)--(35);
			\draw[color=black](32)--(36);
                \draw[color=black](33)--(35);
                \draw[color=black](35)--(36);
                \draw[color=black](37)--(36);
                \draw[color=black](35)--(37);
                \draw[color=black](34)--(36);
            \end{scope}
		\end{tikzpicture}
		\caption{Some graphs with $A$-Wronskian vertex}
\end{figure}
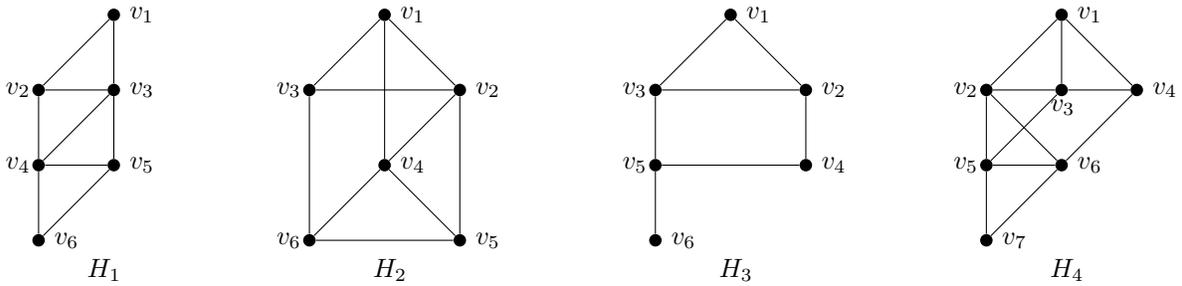

\begin{example}
    We can verify that each vertex of $H_i$ is the $A$-Wronskian vertex of $H_i$ for $i=1,2,3,4$. Let $G$ be an $A$-separable graph. By Corollary \ref{cor1}, we have that $G\circ H_i$, with a vertex of $H_i$ arbitrarily selected as the root vertex, is also $A$-separable for $i=1,2,3,4$.
\end{example}

Next, we aim to characterize the \(M\)-Wronskian vertex of \(H\) using Cauchy's Interlacing Theorem and a lemma from Fisk's work \cite{fisk2008polynomials}. 

Cauchy's Interlacing Theorem describes the eigenvalue relationship between a matrix and its principal submatrices. The theorem is stated as follows:
\begin{theorem}[Cauchy's Interlacing Theorem]\label{cthm}
    Let \(A\) be an \(n\)-order real symmetric matrix, and let \(B\) be an \((n-1)\)-order principal submatrix of \(A\). Denote the eigenvalues of \(A\) as \(\lambda_1 \leqslant \lambda_2 \leqslant \cdots \leqslant \lambda_n\) and those of \(B\) as \(\mu_1 \leqslant \mu_2 \leqslant \cdots \leqslant \mu_{n-1}\). Then, the eigenvalues satisfy the interlacing property:
    $$
    \lambda_1 \leqslant \mu_1 \leqslant \lambda_2 \leqslant \mu_2 \leqslant \cdots \leqslant \mu_{n-1} \leqslant \lambda_n.
    $$
\end{theorem}
The notation \(f \dotle g\) is used to indicate that the roots of \(f(x)\) and \(g(x)\) satisfy the following interleaving property:
\(a_1 < b_1 < a_2 < b_2 < \cdots < a_n\), where \(a_1, \dots, a_n\) are the roots of \(f(x) = 0\) and \(b_1, \dots, b_m\) are the roots of \(g(x) = 0\).

\begin{lem}\cite{fisk2008polynomials}\label{flem}
    Given two polynomials $f$ and $g$ of degree $n$ and $n-1$ with positive leading coefficients, let $a_1 < a_2 < \cdots < a_n$ be the roots of \(f(x) = 0\). Suppose 
    $$
    g(x)=c_1\frac{f(x)}{x-a_1}+\cdots+c_n\frac{f(x)}{x-a_n},
    $$
    where all \(c_i > 0\). Then \(f \dotle g\).
\end{lem}

If $u$ is the $M$-Wronskian vertex of $H$, according to the definition, we always have $\gcd(\phi(M(H),x), \phi(M^u(H),x))=1$. Subsequently, we will prove that this is actually the necessary and sufficient condition for $u$ to be the $M$-Wronskian vertex of $H$.

\begin{theorem}\label{nthm1}
    Let $H$ be a rooted graph with root vertex $u$. Then $u$ is the $M$-Wronskian vertex of $H$ if and only if $\gcd(\phi(M(H),x), \phi(M^u(H),x))=1$.
\end{theorem}
\begin{proof}
First, assume that $u$ is the $M$-Wronskian vertex of $H$. By definition, we have
$$
\phi(M(H), x) \phi'(M^u(H), x)-\phi'(M(H), x)  \phi(M^u(H), x) \neq 0
$$
for all real $x$. This implies that $\gcd(\phi(M(H),x),\phi'(M(H),x))=1$ and $\gcd(\phi(M(H),x),\phi(M^u(H),x))=1$.

Conversely, assume that $\gcd(\phi(M(H),x), \phi(M^u(H),x)) = 1$. Then we have that $H$ is $M$-separable by Theorem \ref{cthm}. By $H$ is $M$-separable, we can assume the eigenvalues of $M(H)$ be $\lambda_1<\lambda_2<\cdots<\lambda_m$ and the eigenvalues of $M^u(H)$ be $\mu_1\leqslant\mu_2\leqslant\cdots\leqslant\mu_{m - 1}$. And according to $\gcd(\phi(M(H),x), \phi(M^u(H),x))=1$, we have $\lambda_i\neq \mu_j$ for any $i=1,2\cdots,m,\ j=1,2,\cdots,m-1$. Thus by  Theorem \ref{cthm} we have 
    $$
\lambda_1<\mu_1<\lambda_2<\mu_2<\cdots<\mu_{m - 1}<\lambda_m,
$$
which implies that $\phi(M(H),x)\ \dotle \ \phi(M^u(H),x)$. Then by Lemma \ref{flem} we can write $$\phi(M^u(H),x)=\sum_{i=1}^m a_i\frac{\phi(M(H),x)}{x-\lambda_i},$$ where $a_i>0$ for $i=1,2,\cdots,m$. Now, consider the derivative of $h(x)=\frac{\phi(M^u(H),x)}{\phi(M(H),x)}$.

$$h'(x)=\frac{  \phi(M(H), x) \phi'(M^u(H), x)-\phi'(M(H), x)  \phi(M^u(H), x)}{(\phi(M(H),x))^2}=-\sum_{i=1}^m\frac{a_i}{(x-\lambda_i)^2}<0.$$
Hence, $\phi(M(H), x) \phi'(M^u(H), x)-\phi'(M(H), x)  \phi(M^u(H), x)<0$  for $x\notin \spec_M(H)$. 

Moreover, since $\gcd(\phi(M(H),x), \phi(M^u(H),x))=1$ and $H$ is $M$-separable, for any $\lambda_{i}\in \spec_M(H)$,  we have$$\phi(M(H), x) \phi'(M^u(H), x)-\phi'(M(H), x)  \phi(M^u(H), x)=-\phi'(M(H), \lambda_i)  \phi(M^u(H), \lambda_i)\neq 0.$$
Combining these results, we conclude that $$\phi(M(H), x) \phi'(M^u(H), x)-\phi'(M(H), x)  \phi(M^u(H), x)\neq 0$$ for all real $x$, then $u$ is the $M$-Wronskian vertex of $H$.
\end{proof}
\begin{remark}\label{re1}
Let \( G \) be a graph of order \( n \) with vertex set \( V(G) = \{u_1, u_2, \dots, u_n\} \). According to Corollary 4.2 in \cite[Page 60]{godsil2017algebraic}, the following identity holds:  
\[
\phi'(A(H), x) \phi(A^{u_i}(H), x)-\phi(A(H), x) \phi'(A^{u_i}(H), x)   = \sum_{k=1}^n \phi_{ik}(G, x)^2,
\]  
where \(\phi_{ij}(G, x)\) denotes the \(ij\)-entry of \(\operatorname{adj}(xI - A(G))\). This implies that  
\[
\phi(A(H), x) \phi'(A^{u_i}(H), x) - \phi'(A(H), x) \phi(A^{u_i}(H), x) \leq 0 \quad \text{for all } x \in \mathbb{R}.
\]
\end{remark}

Furthermore, when $C=I_m$, $G\circ_C H$ turns into the Cartesian product, we have the following result:
\begin{theorem}
  Let $G$ and $H$ be two graphs of order $n$ and $m$, respectively. $K=\{\lambda_1-\lambda_2|\lambda_1,\lambda_2\in \mathrm{Spec}_M(H), \lambda_1\neq \lambda_2\}$, $L=\{\mu_1-\mu_2|\mu_1,\mu_2\in \mathrm{Spec}_M(G),\mu_1\neq\mu_2\}$. Then $G\circ_{I_m} H$ is $M$-separable if and only if both of the following two conditions hold
    \begin{enumerate}[(1).]
        \item $G$ and $H$ are $M$-separable,
        \item $ K\cap L=\emptyset$.
    \end{enumerate}
\end{theorem}

\begin{proof}
    Let $B(\mu)=M(H)+\mu I_n$, then $\mathrm{Spec}(B(\mu))=\{\lambda+\mu|\lambda\in \mathrm{Spec}_M(H)\}$. Assume that conditions (1) and (2) hold. Then $B(\mu)$ has distinct eigenvalues for $\mu\in\mathbb{R}$. By Theorem \ref{1thm1}, it suffices to show that for any $\mu_1,\mu_2\in \mathrm{Spec}_M(G)$, $\mu_1\neq \mu_2$, $\mathrm{Spec}(B(\mu_1))\cap \mathrm{Spec}(B(\mu_2))=\emptyset$. Suppose, for contradiction, that there exist $\mu_1,\mu_2\in \mathrm{Spec}_M(G),$ such that $\mu_1\neq \mu_2$ and $\mathrm{Spec}(B(\mu_1))\cap \mathrm{Spec}(B(\mu_2))\neq \emptyset$, then $\mu_1-\mu_2\in K$. This contradicts condition (2), which ensures that for any $\mu_1,\mu_2\in \mathrm{Spec}_M(G)$, $\mu_1\neq \mu_2$, $\mathrm{Spec}(B(\mu_1))\cap \mathrm{Spec}(B(\mu_2))=\emptyset$.

    Conversely, assume that $G\circ_{I_m} H$ is $M$-separable. By Theorem \ref{1thm1}, condition (1) holds. Suppose, for contradiction, that there exists $a\in  K\cap L$, then assume $\lambda_{i_1}-\lambda_{i_2}=\mu_{j_1}-\mu_{j_2}=a$ where $\lambda_{i_1},\lambda_{i_2}\in \mathrm{Spec}_M(H),\mu_{j_1},\mu_{j_2}\in \mathrm{Spec}_M(G)$. We have $\lambda_{i_1}+\mu_{j_2}=\lambda_{i_2}+\mu_{j_1}\in \mathrm{Spec}(B(\mu_{j_1}))\cap \mathrm{Spec}(B(\mu_{j_2}))$, then $\lambda_{i_2}+\mu_{j_1}$ is a multiple root of $G\circ_{I_m} H$. This contradicts that $G\circ_{I_m} H$ is $M$-separable. Thus, condition (2) holds, completing the proof.
\end{proof}

The proofs of the results in Sections 2 and 3 do not rely on the edge weights of the graphs being equal to $1$. Hence, the corresponding results are valid for weighted graphs. Combining Theorem \ref{nthm1} and Lemma \ref{thm1}, we derive the following result:
\begin{theorem}\label{nthm2}
 Let $G$ and $H$ be two weighted graphs, and $u$ be the root vertex of $H$. Then $G\circ H$ is $M$-separable if and only if $G$ is $M$-separable and $u$ is $M$-Wronskian vertex of $H$.
\end{theorem}
\begin{proof}
    If $G\circ H$ is $M$-separable, from Theorem \ref{nthm1} and Lemma \ref{thm1}, $G$ is $M$-separable. Additionally, we have $(\phi(M(H),x), \phi(M^u(H),x))=1$, then by Theorem \ref{nthm1}, $u$ is $M$-Wronskian vertex of $H$.
    
    If $G$ is $M$-separable and $u$ is $M$-Wronskian vertex of $H$, then from Corollary \ref{cor1}, $G\circ H$ is $M$-separable.
\end{proof}


\section{$M$-Wronskian vertex of $G_v^n$}
Recall that $G_v^n$ is the graph obtained from $G$ by adding a pendant path of length $n$ at $v$ and denote $u_n$ as the pendant vertex of the pendant path. Consider the universal adjacency matrix $U(G_v^n)=aA(G_v^n)+dD(G_v^n)$, where $a,d\in \mathbb{R}$. Let $f_n(x)$ and $g_n(x)$ be the characteristic polynomial of $U(G_v^n)$ and $U^{u_n}(G_v^n)$, respectively. And we stipulate that $g_0(x)$ is the characteristic polynomial of $U^v(G)$. Then by Laplace expansion, we have
\begin{align*}
    g_n(x)&=(x - 2d)g_{n - 1}(x)-a^2g_{n - 2}(x)\ \  \text{for}\ n\geq 2,\\
    f_n(x)&=(x - d)g_n(x)-a^2g_{n - 1}(x)\ \ \text{for}\ n\geq 1.
\end{align*}
Now we prove the following theorem.

\begin{theorem}\label{nthm3}
    If $u_1$ is the $U$-Wronskian vertex of $G_v^1$, then $u_n$ is the $U$-Wronskian vertex of $G_v^n$.
\end{theorem}

\begin{proof}
By Theorem \ref{nthm1}, since $u_1$ is the $U$-Wronskian vertex of $G_v^1$, we have $\gcd(f_1(x),g_1(x)) = 1$. Now we only need to prove that $\gcd(f_n(x),g_n(x)) = 1$ for $n\geq 2$.

We proceed by induction on $n$. From the recurrence relations of $f_n(x)$ and $g_n(x)$, for $n=2$ we have
$$
\gcd(f_2(x),g_2(x))=\gcd(g_1(x),g_2(x))=\gcd(g_1(x),f_1(x)-d g_1(x))=\gcd(g_1(x),f_1(x)) = 1.
$$
Suppose $n\geq 3$ and assume that $\gcd(f_{n - 1}(x),g_{n - 1}(x)) = 1$. Then, we have 
\begin{align*}
    \gcd(f_n(x),g_n(x))&=\gcd((x - d)g_n(x)-a^2g_{n - 1}(x),g_n(x))\\
    &=\gcd(g_{n - 1}(x),g_n(x))\\
    &=\gcd(g_{n - 1}(x),f_{n - 1}(x)-d g_{n - 1}(x))\\
    &=\gcd(f_{n - 1}(x),g_{n - 1}(x)) = 1.
\end{align*}
Thus, by the principle of induction, we have completed the proof.
\end{proof}

Theorem \ref{nthm3} actually covers all the cases where $M\in \{A,Q,L,A_\alpha,U\}$, then we get
\begin{cor}\label{ncor1}
    If $u_1$ is the $M$-Wronskian vertex of $G_v^1$, then $u_n$ is the $M$-Wronskian vertex of $G_v^n$.
\end{cor}
Similar to Theorem \ref{nthm2}, Corollary \ref{ncor1} also holds for weighted graphs.

\begin{example}
    
From $\phi(U(P_2),x)=(x-d)^2-a^2$ and $\phi(U^v(P_2),x)=x-d$ where $v$ is a vertex of $P_2$, we can deduce that  $$\gcd(\phi(U(P_2),x),\phi(U^u(P_2),x))=1$$ for $a\neq 0$, This implies that every vertex of $P_2$ is the $U$-Wronskian vertex of $P_2$, then by Theorem \ref{nthm3}, if $u$ is a pendant vertex of $P_n$, then $u$ is also the $M$-Wronskian vertex of $P_n$.
\end{example}
\begin{example}
Using Corollary \ref{ncor1} in $H_3$, we have that $u_n$ is the $A$-Wronskian vertex of $(H_3)_{v_6}^n$ for $n\geq 1$.
\end{example}
\begin{example}\label{ex4}
Let $G_1 (n,k)$ be the connected graph of order $n$ with diameter $n-2$, which can be obtained by taking a path $P_{n-1}$ with $n-1$ vertices enumerated by $v_1, v_2, \cdots, v_{n-1}$ sequentially and by connecting vertex enumerated $v_n$ to a vertex enumerated $v_k$. Then by computation, $\phi_A(G_1(4,3),x)=x^4-3x^2+1$, $\phi_A(G_1(5,3),x)=x^5-4x^3+2x$, we can verify that $v_4$ is the $A$-Wronskian vertex of $G_1(5,3)$, then by Corollary \ref{ncor1}, we have $v_{n-1}$ is the $A$-Wronskian vertex of $G_1(n,3)$ for $n\geq 5$.
\end{example}
As is demonstrated in Examples 3 and 4, we have that $(H_3)_{v_6}^n$ and $G_1(n,3)$ are $A$-separable graphs. Let $u_n$, $v_{n-1}$ be the root vertex of $(H_3)_{v_6}^n$, $G_1(n,3)$, respectively, and $G$ be an $A$-separable graph, then $G\circ (H_3)_{v_6}^n$ and $G\circ G_1(n,3)$ are $A$-separable.

\begin{example}
By computation, 
\begin{align*}
    \phi(A_{\alpha}(H_5),x)=&x^6 - 16\alpha x^5 + (96\alpha^2 + 16\alpha - 8) x^4 + (-264\alpha^3 - 178\alpha^2 + 98\alpha - 6)x^3\\
    &+ (320\alpha^4 + 660\alpha^3 - 355\alpha^2 + 6\alpha + 8)x^2 + (-126\alpha^5 - 908\alpha^4 + 352\alpha^3 + 154\alpha^2 - 72\alpha + 6)x \\
    &+ 336\alpha^5 + 63\alpha^4 - 282\alpha^3 + 111\alpha^2 - 12\alpha,\\
    \phi(A_{\alpha}^{v_6}(H_5),x)=&x^5 - 15 \alpha x^4 + (82\alpha ^2 + 14\alpha  - 7)x^3 + (-194\alpha ^3 - 140\alpha ^2 + 79\alpha  - 6)x^2\\
    &+ (174\alpha ^4 + 434\alpha ^3 - 253\alpha ^2 + 20\alpha  + 3)x - 21\alpha ^5 - 406\alpha ^4 + 211\alpha ^3 + 18\alpha ^2 - 20\alpha  + 2.
\end{align*}
Using Mathematica, we analyzed the bivariate polynomial 
$$
W(x, \alpha) = W(\phi(A_{\alpha}(H_5), x), \phi(A_{\alpha}^{v_6}(H_5), x)).
$$ 
Within the interval $\alpha \in [0, 1)$, there are exactly eight specific values of $\alpha$, including $\alpha = \frac{2}{3}$, such that $W(x, \alpha) = 0$ for some $x \in \mathbb{R}$. For all other $\alpha$ values in $[0, 1)$, the polynomial satisfies $W(x, \alpha) < 0$ for every $x \in \mathbb{R}$. This implies that $v_6$ is an $A_{\alpha}$-Wronskian vertex of $H_5$ for all $\alpha$ except those eight values. Subsequently, according to Corollary \ref{ncor1},it follows that $u_n$ is the $A_{\alpha}$-Wronskian vertex of $(H_5)_{v_6}^n$ for all $\alpha$ except for those eight values.

\end{example}
\begin{example}
By computation, 
\begin{align*}
    \phi(L(H_6),x)&=x^{6} - 16x^{5} + 95x^{4} - 256x^{3} + 305x^{2} - 126x,\\
    \phi(L^{v_6}(H_6),x)&=x^{5} - 15x^{4} + 81x^{3} - 186x^{2} + 159x - 21.
\end{align*}
We can verify that $v_6$ is the  $L$-Wronskian vertex of $H_6$, then by Corollary \ref{ncor1}, we have $u_n$ is the $L$-Wronskian vertex of $(H_6)_{v_6}^n$ for $n\geq 1$.

\end{example}
\begin{figure}[h]
\centering
\begin{tikzpicture}[scale=2,
            Node/.style={fill,circle,scale=.5}
            ]

        \begin{scope}
    \node[Node,label={[label distance=0pt]0: $v_{1}$}] (21)  at (4, 2.5) {};
    \node[Node,label={[label distance=0pt]0: $v_{2}$}] (22)  at (4.5,2 ) {};
    \node[Node,label={[label distance=-2pt,left]0: $v_{3}$}] (23)  at (3.5, 2) {};
    \node[Node,label={[label distance=0pt]0: $v_{4}$}] (24)  at (4.5, 1.5) {};
    \node[Node,label={[label distance=-2pt,left]0: $v_{5}$}] (25)  at (3.5, 1.5) {};
    \node[Node,label={[label distance=0pt]0: $v_{6}$},fill=red] (26)  at (3.5, 1) {};
    \node[label={[label distance=0pt]0: $H_5$}] (27)  at (3.8, 0.8) {};
    \draw[color=black](21)--(22);
    \draw[color=black](21)--(23);
    \draw[color=black](22)--(23);
    \draw[color=black](23)--(25);
    \draw[color=black](24)--(22);
    \draw[color=black](24)--(25);
    \draw[color=black](26)--(25);
    \draw[color=black](23)--(24);
\end{scope}
        \begin{scope}[xshift=1cm]
 \node[Node,label={[label distance=0pt]0: $v_{1}$}] (31)  at (6, 2.5) {};
    \node[Node,label={[label distance=0pt]0: $v_{2}$}] (32)  at (6.5,2 ) {};
    \node[Node,label={[label distance=-2pt,left]0: $v_{3}$}] (33)  at (5.5, 2) {};
    \node[Node,label={[label distance=0pt]0: $v_{4}$}] (34)  at (6.5, 1.5) {};
    \node[Node,label={[label distance=-2pt,left]0: $v_{5}$}] (35)  at (5.5, 1.5) {};
    \node[Node,label={[label distance=0pt]0: $v_{6}$},fill=red] (36)  at (5.5, 1) {};
    \node[label={[label distance=0pt]0: $H_6$}] (37)  at (5.8, 0.8) {};
    \draw[color=black](31)--(32);
    \draw[color=black](31)--(33);
    \draw[color=black](32)--(33);
    \draw[color=black](33)--(35);
    \draw[color=black](34)--(32);
    \draw[color=black](34)--(35);
    \draw[color=black](36)--(35);
    \draw[color=black](32)--(35);
    \end{scope}
\end{tikzpicture}
\caption{Two graphs with $M$-Wronskian vertex}
\end{figure}
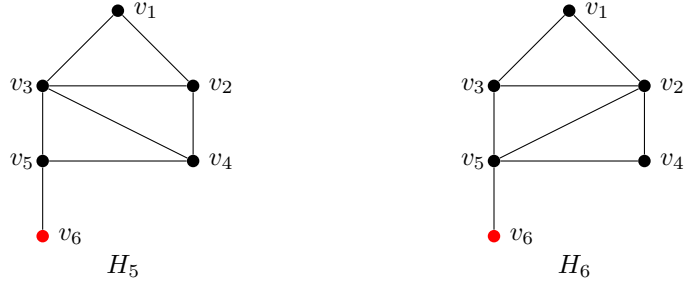

\section{$DMS$-property of $G\circ H$}

In this section, we investigate the $DMS$-property of a rooted product graph and construct infinite pairs of non-isomorphic $M$-cospectral graphs in $\mathcal{G}^M$ based on the results obtained in Sections 3 and 4. By Corollary \ref{thm:spec_rg}, we have the following result.
\begin{lem}\label{lem2}
    Let $G_1$ and $G_2$ be two $M$-cospectral graphs. $H$ is a rooted graph with root vertex $u$. Then
    $G_1\circ H$ and $G_2\circ H$ are $M$-cospectral. 
\end{lem}
By the construction of $G\circ H$, we can immediately obtain that
\begin{lem}\label{lem3}
    Let $G_1$ and $G_2$ be two $M$-cospectral graphs. $H$ is a rooted graph with root vertex $u$. Then
    $G_1\circ H$ and $G_2\circ H$ are isomorphic if and only if $G_1$ and $G_2$ are isomorphic.
\end{lem}

\begin{theorem}\label{thm3} 
Let $H_v^m$ be the rooted graph whose root vertex is $u_m$ for $m\geq 1$. Suppose $G_1$ and $G_2$ are two $M$-cospectral graphs of order $n$. 
If $u_1$ is the $M$-Wronskian vertex of $H_v^1$ and $G_1$, $G_2$ are $M$-separable, then for every $m\geq1$,  $G_1\circ H_v^m$ and $G_2\circ H_v^m$ are $M$-cospectral and $M$-separable.
\end{theorem}
\begin{proof}
    Let $\lambda_1,\lambda_2,\cdots,\lambda_n$ be the common $M$-eigenvalues of $G$ and $H$. According to Corollary \ref{thm:spec_rg}, $$\mathrm{Spec}_M(G_1\circ H_v^m)=\bigcup_{i=1}^n S(\lambda_i)=\mathrm{Spec}_M(G_2\circ H_v^m).$$ 
    where $S(\mu)$ is defined as the set of the roots of $\phi(M(H_v^m),x)-\mu \phi(M^{u_m}(H_v^m),x)$, then $G_1\circ H_v^m$ and $G_2\circ H_v^m$ are $M$-cospectral. Since $u_1$ is the $M$-Wronskian vertex of $H_v^1$, by Corollary \ref{ncor1} we have that $u_m$ is the $M$-Wronskian vertex of $H_v^m$. Then by Corollary \ref{cor1} $G_1\circ H_v^m$ and $G_2\circ H_v^m$ are $M$-separable.
\end{proof}

Applying Theorem \ref{thm3}, we can recursively construct infinite pairs of non-isomorphic $M$-cospectral graphs in $\mathcal{G}^M$.
\begin{example}
    $H_7$ and $H_8$ are a pair of $Q$-cospectral $Q$-separable graphs with $Q$-characteristic polynomial 
    $$\phi_Q(H_7,x)=\phi_Q(H_8,x)=x^6 - 16x^5 + 97x^4 - 282x^3 + 404x^2 - 256x + 48.$$
    And we have \begin{align*}
        \phi(Q(H_5),x)&=x^6 - 16x^5 + 96x^4 - 276x^3 + 396x^2 - 262x + 60,\\
        \phi(Q^{v_6}(H_5),x)&=x^5 - 15x^4 + 82x^3 - 206x^2 + 238x - 101.
    \end{align*}
Then $\gcd( \phi(Q(H_5),x), \phi(Q^{v_6}(H_5),x))=1$, by Corollary \ref{ncor1}, we have that $u_n$ is the $Q$-Wronskian vertex of $(H_5)_{v_6}^n$ for $n\geq 1$. Thus, we set $u_n$ as the root vertex of $(H_5)_{v_6}^n$, by Theorem \ref{thm3}, $H_7\circ (H_5)_{v_6}^n$ and $H_8\circ (H_5)_{v_6}^n$ are non-isomorphic $Q$-cospectral graphs in $\mathcal{G}^Q$ for $n\geq 1$.
    
\end{example}

\begin{figure}[h]
\centering
		\begin{tikzpicture}[scale=2,
			Node/.style={fill,circle,scale=.5}
			]
            \begin{scope}
		\node[Node,label={[label distance=-2pt,left]0: $v_{1}$}] (0)  at (1, 2.5) {};
        \node[Node,label={[label distance=0pt]0: $v_{2}$}] (1)  at (2,2.5 ) {};
        \node[Node,label={[label distance=-2pt,left]0: $v_{3}$}] (2)  at (1, 1.75) {};
        \node[Node,label={[label distance=0pt]0: $v_{4}$}] (3)  at (2, 1.75) {};
        \node[Node,label={[label distance=-2pt,left]0: $v_{5}$}] (4)  at (1, 1) {};
        \node[Node,label={[label distance=0pt]0: $v_{6}$}] (5)  at (2, 1) {};
        \node[label={[label distance=0pt]0: $H_7$}] (6)  at (1.3, 0.8) {};
			\draw[color=black](0)--(1);
			\draw[color=black](0)--(2);
			\draw[color=black](0)--(3);
			\draw[color=black](1)--(3);
			\draw[color=black](2)--(4);
			\draw[color=black](2)--(5);
			\draw[color=black](3)--(4);
			\draw[color=black](3)--(5);

\end{scope}
        \begin{scope}[xshift=1cm]
        \node[Node,label={[label distance=-2pt,left]0: $v_{1}$}] (11)  at (3.5, 2.5) {};
        \node[Node,label={[label distance=0pt]45: $v_{2}$}] (12)  at (3.5,2 ) {};
        \node[Node,label={[label distance=-2pt,left]0: $v_{3}$}] (13)  at (3, 2) {};
        \node[Node,label={[label distance=0pt]0: $v_{4}$}] (14)  at (4, 2) {};
        \node[Node,label={[label distance=-2pt,left]0: $v_{5}$}] (15)  at (3.5, 1.5) {};
        \node[Node,label={[label distance=0pt]0: $v_{6}$}] (16)  at (3.5, 1) {};
        \node[label={[label distance=0pt]0: $H_8$}] (17)  at (3.3, 0.8) {};
			\draw[color=black](11)--(12);
			\draw[color=black](11)--(13);
			\draw[color=black](11)--(14);
			\draw[color=black](12)--(13);
			\draw[color=black](12)--(14);
			\draw[color=black](13)--(15);
			\draw[color=black](14)--(15);
			\draw[color=black](16)--(15);  
            \end{scope}
		\end{tikzpicture}
		\caption{A pair of $Q$-cospectral $Q$-separable graphs}
\end{figure}
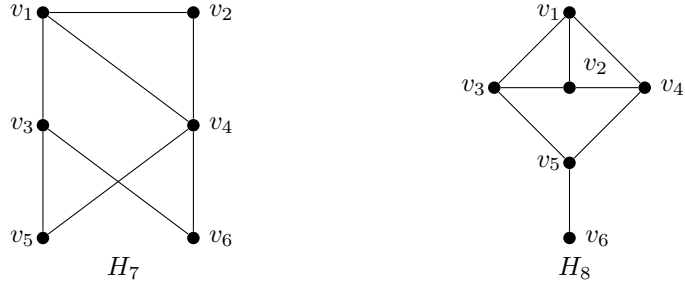

Next, we present two results related to the $DMS$ problem.

\begin{theorem}
    If $G\circ H$ is $DMS$, then $G$ is $DMS$.
\end{theorem}
\begin{proof}
    Let $G_1$ be any graph such that $\mathrm{Spec}_M(G_1)=\mathrm{Spec}_M(G)$. By Lemma \ref{lem2}, we obtain $\mathrm{Spec}_M(G\circ H)=\mathrm{Spec}_M(G_1\circ H)$, then $G\circ H\cong G_1\circ H$ because $G\circ H$ is $DMS$, hence $G\cong G_1$ by Lemma \ref{lem3}.
\end{proof}

Conversely, if $G$ is $DMS$, we cannot determine whether $G\circ H$ is $DMS$. Regarding this issue, we have a partial result.
\begin{theorem}
     Let $H$ be the rooted graph with root vertex $u$, and $\gcd(\phi(M(H),x),\phi(M^u(H),x))=1$. If $G_1\circ H$ and $G_2\circ H$ are two $M$-cospectral graphs and $G_1$ is $DMS$, then $G_1\circ H\cong G_2\circ H$.
\end{theorem}
\begin{proof}
   By Lemma \ref{lem3}, we only need to show that $G_1$ and $G_2$ are $M$-cospectral. Suppose the contrary that $\mathrm{Spec}_M(G_1)\neq \mathrm{Spec}_M(G_2)$. Let $\mathrm{Spec}_M(G_1)=\{\lambda_1, \lambda_2,\cdots,\lambda_n\}$ and $\mathrm{Spec}_M(G_2)=\{\lambda_1', \lambda_2',\cdots,\lambda_n'\}$. Then there exists $\lambda_i\in \mathrm{Spec}_M(G_1)$ such that $\lambda_i\neq \lambda_j'$ for $1\leq j\leq n$. Denote $S(\mu)$ as the set of the roots of $\phi(M(H),x)-\mu \phi(M^u(H),x)$. Since $\gcd(\phi(M(H),x),\phi(M^u(H),x))=1$, we have $S(\lambda_i)\bigcap S(\lambda_j') =\emptyset$ for $1\leq j\leq n$. Then by Corollary \ref{thm:spec_rg}, $\mathrm{Spec}_M(G_1\circ H)\neq \mathrm{Spec}_M(G_2\circ H)$, a contradiction. Therefore, we have $G_1\cong G_2$ by the fact that $G_1$ is $DMS$. Subsequently, by Lemma \ref{lem2}, we can conclude that $G_1\circ H\cong G_2\circ H$.
\end{proof}

\section{$M$-controllable graphs}
Recall that a $M$-controllable graph is a connected graph whose $M$-matrix has distinct
and main eigenvalues. The necessary and sufficient condition for $G\circ H$ to have distinct $M$-eigenvalues is shown in Section 3. We give the necessary and sufficient condition for $G\circ H$ to be $M$-controllable in this section.

\begin{theorem}\label{conthm1}
Let $G$ be a graph of order $n$, and $H$ be a rooted graph of order $m$ with root vertex $u$, where $u$ corresponds to the first row and first column of $M(H)$. Denote $B(\mu) = M(H)+\mu E_{1,1}$. Then, $G\circ H$ is $M$-controllable if and only if $G$ is $M$-controllable, $(\phi(M(H),x),\phi(M^u(H),x)) = 1$ and for any $\mu\in\mathrm{Spec}_M(G)$, $B(\mu)=M(H)+\mu E_{1,1}$ is controllable.
\end{theorem}

\begin{proof}
Let $\mathrm{Spec}_M(G)=\{\mu_1,\mu_2,\cdots,\mu_n\}$, and $\xi_1,\xi_2,\cdots,\xi_n$ form an orthonormal basis of $M(G)$ such that $M(G)\xi_i=\mu_i\xi_i$ for $i = 1,2,\cdots,n$. Let the spectrum of $B(\mu_j)$ be $\mathrm{Spec}(B(\mu_j))=\{\lambda_{1j},\lambda_{2j},\cdots,\lambda_{mj}\}$, and $\eta_{1j},\eta_{2j},\cdots,\eta_{mj}$ form an orthonormal basis of $B(\mu_j)$ such that $B(\mu_j)\eta_{ij}=\lambda_{ij}\eta_{ij}$ for $i = 1,2,\cdots,m$.

If $G\circ H$ is $M$-controllable, then $G\circ H$ is $M$-separable. From Theorem $\ref{nthm2}$, we have $G$ is $M$-separable and $(\phi(M(H),x),\phi(M^u(H),x))=1$. According to Corollary $\ref{thm:spec_rg}$, $\{\eta_{ij}\otimes\xi_j\}$ spans the eigenspace of $M(G\circ H)$. Note that \begin{equation}
    \boldsymbol{1}_{mn}^T(\eta_{ij}\otimes\xi_j)=(\boldsymbol{1}_m^T\otimes\boldsymbol{1}^T_n)(\eta_{ij}\otimes\xi_j)=(\boldsymbol{1}^T_m\eta_{ij})\otimes(\boldsymbol{1}_n^T\xi_j)=(\boldsymbol{1}^T_m\eta_{ij})(\boldsymbol{1}_n^T\xi_j).\label{eq1}
\end{equation}
 From this, by $\eta_{ij}\otimes\xi_j$ is the main eigenvector of $M(G\circ H)$, we have $\eta_{ij}$ is the main eigenvector of $B(\mu_j)$ and $\xi_j$ is the main eigenvector of $M(G)$. Thus, for any $\mu\in\mathrm{Spec}_M(G)$, $B(\mu)=M(H)+\mu E_{1,1}$ is controllable and $G$ is $M$-controllable.

If $G$ is $M$-controllable, $(\phi(M(H),x),\phi(M^u(H),x)) = 1$ and for any $\mu\in\mathrm{Spec}_M(G)$, $B(\mu)=M(H)+\mu E_{1,1}$ is controllable, we have $G\circ H$ is $M$-separable from Theorem $\ref{nthm2}$. And by equation \eqref{eq1}, we can conclude that all the eigenvectors of $M(G\circ H)$ are main eigenvectors, thus $G\circ H$ is $M$-controllable.
\end{proof}

Then we can give a sufficient condition for $G\circ H$ to be $M$-controllable according to Theorem \ref{conthm1}. 
\begin{cor}\label{concor}
    Let $G$ be a graph of order $n$, and $H$ be a rooted graph of order $m$ with root vertex $u$, where $u$ corresponds to the first row and first column of $M(H)$. Denote $B(\mu) = M(H) + \mu E_{1,1} $. Then $G\circ H$ is $M$ controllable if $G$ is $M$ controllable, $(\phi(M(H),x), \phi(M^u(H),x))=1$ and for any $\mu\in \mathbb{R}$, $B(\mu) = M(H) + \mu E_{1,1} $ is controllable.
\end{cor}

From Theorem \ref{conthm1}, the following result is obvious.
\begin{cor}\label{concor1}
    Let $G$ be a graph of order $n$, and $H$ be a rooted graph of order $m$ with root vertex $u$. If $G\circ H$ is $M$-controllable and $G$ has zero eigenvalues, then both $G$ and $H$ are $M$-controllable.
\end{cor} 
A natural question arises:if $G$ and $H$ are both $M$-controllable, does it necessarily follow that $G \circ H$ is also $M$-controllable? In fact, even in the setting of $A$-controllability, this property does not always hold. For example, consider the graph $G_1(7, 3)$, which is known to be $A$-controllable (see \cite{zbMATH07424199}). If $v_1$ is chosen as the root vertex of $G_1(7, 3)$, the rooted product graph $G_1(7, 3) \circ G_1(7, 3)$ is neither $A$-controllable nor $A$-separable.

Furthermore, even if both $G$ and $H$ are $M$-controllable, and the rooted product graph $G \circ H$ is $M$-separable, this does not guarantee that $G \circ H$ is $M$-controllable. 
A similar situation occurs in the context of $A$-controllability. Consider the rooted product graph $H_9 \circ H_{10}$, where $v_7$ is taken as the root vertex of $H_{10}$. Calculations show that $H_9 \circ H_{10}$ is $A$-separable, but the rank of its walk matrix, $\mathrm{rank}(W(H_9 \circ H_{10}))$, is 48, which is less than the required full rank of 49. This implies that $H_9 \circ H_{10}$ is not $A$-controllable, despite both $H_9$ and $H_{10}$ being $A$-controllable.

On the other hand, if $G\circ H$ is $M$-controllable, we can confirm that $G$ is $M$-controllable from Theorem \ref{conthm1}, and if $G$ has zero eigenvalues, we have $H$ is $M$-controllable from Corollary \ref{concor1}.  But if $G$ has no zero eigenvalues, then whether $H$ is $M$-controllable remains uncertain. For instance, when \(M = A\), if \(G\) is \(A\)-controllable and has no zero eigenvalues, taking a vertex of $P_2$ as the root of \(P_2\) makes \(G\circ P_2\) \(A\)-controllable according to Theorem 4.1 in \cite{zbMATH06680902}. The same holds for \(M = Q\) with Theorem 4.1 in \cite{zbMATH06812655}, and for \(M = A_{\alpha}\) with Theorem 4.2 in \cite{zbMATH07794544}. Therefore, we naturally put forward the following two questions.
\begin{question}
    What kind of graph-theoretic characteristics and conditions should $G$ and $H$ satisfy for the rooted product graph $G\circ H$ to be $M$-controllable?
\end{question}
\begin{question}
  What kind of conditions should $G\circ H $ satisfy for the graphs $G$ and $H$ to be $M$-controllable?
\end{question}

\begin{figure}[H]
\centering
		\begin{tikzpicture}[scale=2,
			Node/.style={fill,circle,scale=.5}
			]

        \begin{scope}
		\node[Node,label={[label distance=0pt]0: $v_{1}$}] (0)  at (1.5, 2.5) {};
        \node[Node,label={[label distance=-2pt,left]0: $v_{2}$}] (1)  at (1,2.0 ) {};
        \node[Node,label={[label distance=0pt]0: $v_{3}$}] (2)  at (1.5, 2.0) {};
        \node[Node,label={[label distance=-2pt,left]0: $v_{4}$}] (3)  at (1, 1.5) {};
        \node[Node,label={[label distance=0pt]0: $v_{5}$}] (4)  at (1.5, 1.5) {};
        \node[Node,label={[label distance=0pt]0: $v_{6}$}] (5)  at (1, 1) {};
        \node[Node,label={[label distance=0pt]0: $v_{7}$}] (6)  at (2, 1.5) {};
        \node[label={[label distance=0pt]0: $H_9$}] (7)  at (1.2, 0.8) {};
			\draw[color=black](0)--(1);
			\draw[color=black](1)--(2);
			\draw[color=black](2)--(3);
			\draw[color=black](3)--(4);
			\draw[color=black](0)--(2);
			\draw[color=black](1)--(3);
			\draw[color=black](2)--(4);
			\draw[color=black](4)--(2);
                \draw[color=black](4)--(5);
                \draw[color=black](3)--(5);
                \draw[color=black](2)--(6);
\end{scope}
        \begin{scope}[xshift=1cm]
        
        \node[Node,label={[label distance=0pt]0: $v_{1}$}] (11)  at (3.5, 2.5) {};
        \node[Node,label={[label distance=-2pt,left]0: $v_{2}$}] (12)  at (3,2.0 ) {};
        \node[Node,label={[label distance=0pt]0: $v_{3}$}] (13)  at (3.5, 2.0) {};
        \node[Node,label={[label distance=-2pt,left]0: $v_{4}$}] (14)  at (3, 1.5) {};
        \node[Node,label={[label distance=0pt]0: $v_{5}$}] (15)  at (3.5, 1.5) {};
        \node[Node,label={[label distance=0pt]0: $v_{6}$}] (16)  at (3, 1) {};
        \node[Node,label={[label distance=-2pt,left]0: $v_{7}$},fill=red] (17)  at (2.5,1.75) {};
        \node[label={[label distance=0pt]0: $H_{10}$}] (18)  at (3.2, 0.8) {};
        \draw[color=black](12)--(11);
        \draw[color=black](13)--(11);
        \draw[color=black](12)--(13);
        \draw[color=black](12)--(17);
        \draw[color=black](12)--(14);
        \draw[color=black](13)--(17);
        \draw[color=black](13)--(15);
        \draw[color=black](14)--(17);
        \draw[color=black](14)--(15);
        \draw[color=black](14)--(16);
    \end{scope}
        \end{tikzpicture}
		\caption{Two $A$-controllable graphs}
\end{figure}
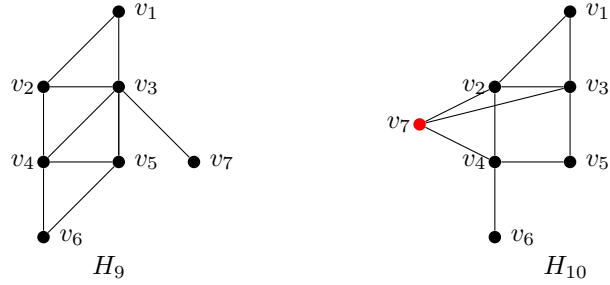

Recall that the notation $\mathcal{G}^A$ represents the set of connected graphs with distinct $A$-eigenvalues. In what follows, we define $\mathcal{G}$ to be the set of connected graphs, $\mathcal{G}^{A^*}$ to be the set of connected $A$-controllable graphs,
$\mathcal{G}^A_W$ as the set of connected graphs with $A$-Wronskian vertex, and $\mathcal{G}^{A^*}_W$ as the set of connected $A$-controllable graphs having $A$-Wronskian vertex, respectively. The following table presents the quantity of such graphs of order not exceeding $9$. Regarding $\vert\mathcal{G}^A\vert$ and $\vert\mathcal{G}^{A^*}\vert$, one can refer to sequences A242952\cite{A242952} and A371897\cite{A371897} in OEIS respectively. Additionally, the values of $|\mathcal{G|,}\vert\mathcal{G}^A_W\vert$ and $\vert\mathcal{G}^{A^*}_W\vert$ were calculated using the Mathematica program.

\begin{table}[H]
  \begin{center}
    \renewcommand{\arraystretch}{1.2} 
    \setlength{\tabcolsep}{6pt} 
    \begin{tabu} to \textwidth {X[1,c]|X[2,c]|X[2,c]|X[2,c]|X[2,c]|X[2,c]} 
      \hline
      \textbf{Order} & \textbf{$|\mathcal{G}|$} & \textbf{$|\mathcal{G}^A|$} & \textbf{$|\mathcal{G}^{A^*}|$}& \textbf{$|\mathcal{G}^A_W|$} & \textbf{$|\mathcal{G}^{A^*}_W|$}\\
      \hline
      1 & 1 & 1 & 1 & 0 & 0\\
      2 & 1 & 1 & 0 & 1 & 0\\
      3 & 2 & 1 & 0 & 1 & 0\\
      4 & 6 & 3 & 0 & 3 & 0\\
      5 & 21 & 11 & 0 & 9 & 0\\
      6 & 112 & 54 & 8 & 37 & 8\\
      7 & 853 & 539 & 85 & 414 & 85\\
      8 & 11117 & 7319 & 2275 & 5984 & 2275\\ 
      9 & 261080 & 209471 & 83034 & 186053 & 83034 \\ 
      \hline
    \end{tabu}
  \end{center}
\end{table}

We have found that when $6\leq n\leq 9$, all the graphs in $\mathcal{G}^{A^*}$ are also in $\mathcal{G}^{A^*}_W$. Therefore, it might also be an interesting problem to study the relationship between $M$-controllability and the possession of $M$-Wronskian vertices.
\bibliographystyle{plain}
\bibliography{reference}

\end{document}